\newif\ifipco
\ipcotrue 
\ipcofalse
\documentclass[11pt]{article}
\usepackage{amssymb}
\usepackage{graphicx}
\usepackage{cite}
\usepackage{fullpage}
\usepackage{color}
\usepackage[usenames,dvipsnames,svgnames,table]{xcolor}

\usepackage{tikz}
\usetikzlibrary{positioning}

\def\visible<#1>{}  

\usepackage[utf8]{inputenc}

\usepackage[english]{babel}
\usepackage{amsfonts}
\usepackage{amsmath}
\usepackage{amsthm}
\usepackage{latexsym}
\usepackage{color}
\usepackage{subfigure}
\usepackage{enumerate}

\usepackage{hyperref}  

\usepackage{ifpdf}

\graphicspath{{figures/}}

\DeclareMathOperator    \conv           {conv}

\DeclareMathOperator    \intr                   {int}
\DeclareMathOperator    \lin                    {lin}

\DeclareMathOperator    \rank                     {rank}
\DeclareMathOperator    \relint         {rel\,int}

\DeclareMathOperator    \verts          {vert}

\newtheorem{theorem}{Theorem}[section]

\makeatletter
\newcommand\MkNewTheorem[2]{%
  \newtheorem{#1}{#2}
  \expandafter\def\csname c@#1\endcsname{\c@theorem}
  \expandafter\def\csname p@#1\endcsname{\p@theorem}
  \expandafter\def\csname the#1\endcsname{\thetheorem}
  \expandafter\def\csname #1name\endcsname{#2}
}

\MkNewTheorem{corollary}{Corollary}
\MkNewTheorem{lemma}{Lemma}
\MkNewTheorem{proposition}{Proposition}
\MkNewTheorem{prop}{Proposition}
\MkNewTheorem{claim}{Claim}
\MkNewTheorem{observation}{Observation}
\MkNewTheorem{obs}{Observation}
\MkNewTheorem{conjecture}{Conjecture}
\MkNewTheorem{openquestion}{Open question}

\theoremstyle{definition}
\MkNewTheorem{example}{Example}
\MkNewTheorem{exercise}{Exercise}
\MkNewTheorem{notation}{Notation}
\MkNewTheorem{assumption}{Assumption}
\MkNewTheorem{definition}{Definition}
\MkNewTheorem{remark}{Remark}
\MkNewTheorem{goal}{Goal}

\newcommand{\old}[1]{{}}

\newcommand{\bb}{\mathbb}

\newcommand{\R}{\bb R}
\newcommand{\Q}{\bb Q}
\newcommand{\Z}{\bb Z}
\newcommand{\N}{\bb N}

\def\ve#1{\mathchoice{\mbox{\boldmath$\displaystyle\bf#1$}}
{\mbox{\boldmath$\textstyle\bf#1$}}
{\mbox{\boldmath$\scriptstyle\bf#1$}}
{\mbox{\boldmath$\scriptscriptstyle\bf#1$}}}

\newcommand{\mqed}{\hfill \qed}

\newcommand{\picomb}{\pi_{\mathrm{comb}}}
\newcommand{\pisym}{\pi_{\mathrm{sym}}}
\newcommand{\tpisym}{\tilde{\pi}_{\mathrm{sym}}}

\newcommand{\pipwl}{\pi_{\mathrm{pwl}}}
\newcommand{\pifill}{\pi_{\mathrm{fill-in}}}

\renewcommand{\P}{\mathcal P}
\newcommand{\ie}{\textit{i.e.}}

\newcommand{\x}{{\ve x}}
\newcommand{\y}{{\ve y}}

\renewcommand{\a}{{\ve a}}

\renewcommand{\b}{{\ve b}}

\renewcommand{\paragraph}[1]{\medskip \noindent \textbf{#1}}

\newcommand{\B}{U}
\newcommand{\remove}[1]{}

\newenvironment{psmallmatrixbig}{\bigl(\smallmatrix}{\endsmallmatrix\bigr)}
\newcommand\InlineFrac[2]{#1/#2}  
\newcommand\ColVec[3][\relax]
{
  \ifx#1\relax
  \bgroup\let\frac=\InlineFrac\begin{psmallmatrixbig}#2\vphantom{/}\\#3\vphantom{/}\end{psmallmatrixbig}\egroup
  \else
  \bgroup\let\frac=\InlineFrac\begin{psmallmatrixbig}\ifx#200\else#2/#1\fi\\\ifx#300\else#3/#1\fi\end{psmallmatrixbig}\egroup
  \fi
}

\title{Minimal cut-generating functions are nearly extreme}

\author{Amitabh Basu\thanks{Dept.~of Applied Mathematics and Statistics, The Johns Hopkins University. A. Basu gratefully acknowledges support from NSF grant CMMI1452820.  }\and Robert Hildebrand\thanks{Institute for Operations Research, Dept. of Mathematics, ETH Z\"urich, Switzerland. Most of this research was conducted while Robert Hildebrand was a postdoctoral researcher at the Institute for Operations Research, Department of Mathematics, ETH Z\"urich.}\and Marco Molinaro\thanks{Computer Science Department, PUC-RIO, Brazil}}

\usepackage[normalem]{ulem}

\begin{document}
 \newcommand{\tgreen}[1]{\textsf{\textcolor {ForestGreen} {#1}}}
 \newcommand{\tred}[1]{\texttt{\textcolor {red} {#1}}}
 \newcommand{\tblue}[1]{\textcolor {blue} {#1}}

\maketitle
\begin{abstract}
	We study continuous (strongly) minimal cut generating functions for the model where all variables are integer. We consider both the original Gomory-Johnson setting as well as a recent extension by Cornu\'ejols and Y{\i}ld{\i}z. We show that for any continuous minimal or strongly minimal cut generating function, there exists an extreme cut generating function that approximates the (strongly) minimal function as closely as desired. In other words, the extreme functions are ``dense'' in the set of continuous (strongly) minimal functions.
\end{abstract}



\section{Introduction} 
\label{sec:introduction}

	Cut-generating functions are an important approach for deriving, understanding, and analyzing general-purpose cutting planes for mixed-integer programs. Given a natural number $n\in \N$ and a closed subset $S \subseteq \R^n\setminus \{0\}$, a {\em cut-generating function (CGF) for $S$}
 is a function $\pi\colon \R^n \to \R$ such that for every choice of natural number $k \in \N$ and $k$ vectors $r_1, \ldots, r_k \in \R^n$, the inequality 
	\vspace{-5pt}
	$$\sum_{i=1}^k\pi(r_i)y_i \geq 1$$
	is valid for the set \begin{equation}\label{eq:pure-int}Q_0 = \bigg\{y \in \Z^k_+: \sum_{i=1}^k r_iy_i \in S\bigg\}.\end{equation} Note that CGFs for $S$ only depend on $n$ and $S$, and should work for all choices of $k \in \N$ and $r_1, \ldots, r_k \in \R^n$. 

Cut-generating functions were originally considered for sets $S$ of the form $S = b + \Z^n$ for some $b \in \R^n \setminus\Z^n$ by Gomory and Johnson~\cite{infinite,infinite2,johnson} under the name of the {\em infinite group relaxation}. Such sets $S$ will be called {\em affine lattices}. In this case the connection with Integer Programming is clear: the set $Q_0$ is the projection of a mixed-integer set in tableau form $Q_{\textrm{rel}} = \{(x,y) \in \Z^{m} \times \Z^k_+ : x = -b + \sum_i r_i y_i\}$ onto the non-basic variables, so CGFs give valid cuts for this set. Notice that $Q_{\textrm{rel}}$ arises as a relaxation of a generic pure integer program in standard form by dropping the non-negativity of the basic variables $x$. So another important setting of CGFs is one that does not involve this relaxation, namely using sets $S$ of the form $S = b + \Z^n_+$, where $b \in \R^n \setminus \Z^n$ and $b \leq 0$~\cite{yildiz2015integer}; this now corresponds to the projection of the ``unrelaxed'' set $Q_{\textrm{full}} = \{(x,y) \in \Z^{m}_+ \times \Z^k_+ : x = -b + \sum_i r_i y_i\}$. We call such sets $S$ \emph{truncated affine lattices}.
  
	Cut-generating functions have received significant attention in the literature (see the surveys~\cite{basu2015geometric,basu2016light,basu2016light2} and the references therein). One important feature is that cut-generating functions capture known general purpose cuts, for example the prominent GMI cuts and more generally split cuts (when $S$ is an affine lattice) and the \emph{lopsided cuts} of Balas and Qualizza \cite{balas2012monoidal} (when $S$ is a truncated affine lattice). Moreover, the CGF perspective gives a clean way of understanding cuts, since they abstract the finer structure of mixed integer sets and only depend on $n$ and $S$ (for affine and truncated affine lattices this is just the shift vector $b$).



\vspace{-0pt}
\paragraph{Strength of cut generating function and extreme functions.} Since their introduction, there has been much interest in understanding what the ``best'' CGFs are -- the ones that cut ``most deeply''. CGFs can be stratified and at the first level we have (strongly) minimal functions. An inequality $\alpha\cdot x \geq \alpha_0$ given by $\alpha \in \R^n$ and $\alpha_0\in \R$ will be called valid for $S$ if every $s \in S$ satisfies $\alpha \cdot s \geq \alpha_0$. A CGF $\pi$ for $S$ is called {\em strongly minimal} if there does not exist a {\em different} CGF $\pi'$, a real number $\beta \geq 0$ and a valid inequality $\alpha\cdot x \geq  \alpha_0$ for $S$, such that for all $r\in \R^n$, $\beta\pi'(r) + \alpha\cdot r \leq \pi(r)$ and $\beta + \alpha_0 \geq 1$~\cite{yildiz2015integer}.\footnote{When $S$ is an affine lattice this notion is equivalent to notion of \emph{minimal inequality} used in the literature.} This definition captures the standard idea of non-dominated inequalities. Gomory and Johnson characterized all continuous strongly minimal functions when $S$ is an affine lattice.

\begin{theorem}[Gomory and Johnson \cite{infinite}] \label{thm:minimal} Let $S =  b + \Z^n$ for some $b \in \R^n\setminus \Z^n$. A continuous function $\pi\colon \R^n \to \R$ is a strongly minimal function if and only if all of the following hold:
\vspace{-5pt}
\begin{itemize}
\item[(i)] $\pi$ is a nonnegative function with $\pi( z) = 0$ for
  all $ z\in \Z^n$,
  \item[(ii)] $\pi$ is subadditive, i.e., $\pi( r_1 +  r_2) \leq \pi( r_1) + \pi( r_2)$ for all $r_1, r_2 \in \R^n$, and 
  \item[(iii)]$\pi$ satisfies $\pi( r) + \pi( b- r) = 1$ for all $ r \in \R^n$. (This condition is known as the symmetry condition.)
 \end{itemize}
\end{theorem}

Note that the first two conditions imply that $\pi$ is periodic modulo $\Z^n$, i.e., $\pi(r) = \pi(r +  z)$ for all $z \in \Z^n$. This characterization was also recently generalized by Cornu\'ejols and Y{\i}ld{\i}z to a wider class of sets $S$, which includes truncated affine lattices as well~\cite{yildiz2015integer} (see Theorem \ref{thm:minimal-2} below). 

A stronger notion than strong minimality is that of an {\em extreme function}. We say that a CGF $\pi$ is extreme if there do not exist distinct CGFs $\pi_1$ and $\pi_2$ such that $\pi= \frac{\pi_1+\pi_2}{2}$. This is a subset of strongly minimal functions~\cite{yildiz2015integer} that corresponds to a notion of ``facets'' in the context of CGFs 
(see also \cite{basu2016light,basu2016light2} for other notions of ``facet'' for CGFs). Because of the importance of facet-defining cuts in Integer Programming, there has been substantial interest in obtaining and understanding extreme functions (see \cite{basu2016light,basu2016light2} for a survey). For example, a celebrated result is Gomory and Johnson's 2-Slope Theorem (Theorem \ref{thm:2slope} below) that gives a \emph{sufficient condition} for a CGF to be extreme (in the affine lattice setting with $n=1$; see \cite{3slope,basu-hildebrand-koeppe-molinaro:k+1-slope,yildiz2015integer} for generalizations). 

Unfortunately the structure of extreme functions seems much more complicated than that of minimal functions. For example, even verifying the extremality of a function is not completely understood (see~\cite{basu2016light,basu2016light2} for preliminary steps in this direction); a simple characterization like Theorem~\ref{thm:minimal} seems all the more unlikely. 


\paragraph{Our results.} As noted above, it is easy to verify that extreme functions are always strongly minimal. We prove an approximate converse: 
in a strict mathematical sense, strongly minimal functions for $n=1$ are ``close'' to being extreme functions. More precisely, in the affine lattice setting we prove the following.
%
	
	\begin{theorem}
	\label{thm:theoremZ}
	Let $S = b + \Z$ for some $b \in \Q\setminus \Z$. Let $\bar{\pi}$ be a continuous strongly minimal function for $S$. Then for every $\epsilon > 0$ there is an extreme ($2$-slope) function $\pi^*$ such that $\|\bar{\pi} - \pi^*\|_{\infty} \le \epsilon$, where $\|\cdot\|_{\infty}$ is the sup norm. 
	\end{theorem}
	Equivalently, this states that extreme functions are dense, under the sup norm, in the set of strongly minimal functions. This surprising property of CGFs relies on their infinite-dimensional nature: for finite-dimensional polyhedra, a (non-facet) minimal inequality can never be \emph{arbitrarily} close (under any reasonable distance) to a facet.
	
	In the truncated affine lattice setting we prove a similar result under an additional assumption.
A function $\phi\colon\R\to\R$ is {\em quasi-periodic (with period $d$)}, if there are real numbers $d>0$ and $c \in \R$ such that $\phi(r + d) = \phi(r) + c$ for every $r \in \R$. All explicitly known CGFs from the literature are quasi-periodic. Moreover, quasi-periodic piecewise linear CGFs can be expressed using a finite number of parameters,  
making them 
attractive from a computational perspective.
	
	
	\begin{theorem}
	\label{thm:theoremZPlus}
		Let $S = b + \Z_+$, where $b \in \Q \setminus \Z$ and $b \leq 0$. Let $\bar{\pi}$ be a continuous, strongly minimal function for $S$ that is quasi-periodic with rational period. Then for any $M \in \R_+$ and $\epsilon > 0$ there is an extreme ($2$-slope) function $\pi^*$ such that $|\bar{\pi}(r) - \pi^*(r)| \le \epsilon$ for all $r \in [-M,M]$. 
	\end{theorem}
	

	Our results imply that for the purpose of cutting-planes, these strongly minimal functions perform essentially as well as extreme functions, at least for the $n=1$ case, i.e., cutting planes from a single row. This points out the limitations of extremality as a measure for the quality of one-dimensional CGFs and suggests the need for alternative measures (see \cite{tspace}). {\em However, we have not been able to establish such results for $n\geq 2$. The question remains open whether extremality is a more useful concept in higher dimensions, making it relevant for multi-row cuts.}

\paragraph{Continuous infinite relaxation.} 	
 While our main results are regarding the infinite group problem, in the last section we consider another well studied model for cutting planes -- the so-called \emph{infinite continuous relaxation}, 
 first introduced in~\cite{BorCor}. We recall here some basic definitions and results about this model. Given a natural number $n\in \N$ and a closed subset $S \subseteq \R^n \setminus\{0\}$, a {\em continuous cut-generating function (CCGF) for $S$} is a function $\psi : \R^n \rightarrow \R$ such that the inequality $$\sum_{i = 1}^k \psi(r^i) s_i \ge 1$$ is valid for

%
\begin{align*}
	C_S(r^1, \ldots, r^k) := \left\{ s \in \R^k_+ : \sum_{i = 1}^k r^i s_i \in S\right\}
\end{align*}
	for every choice of a natural number $k\in \N$ and every set of $k$ vectors $r^1, \ldots, r^k$. Note the contrast with~\eqref{eq:pure-int} where the variables $y_i$ took nonnegative, {\em integer} values, as opposed to the variables $s_i$ which take nonnegative real values in $C_S(r^1, \ldots, r^k)$. This is what is intended by the word ``continuous'' in CCGF; it has nothing to do with continuity in the analytic sense. We will often abbreviate $C_S(r^1, \ldots, r^k)$ as $C_S(R)$ where $R \in \R^{n\times k}$ is the matrix with columns $r^1, \ldots, r^k$.

	We also have the definitions of minimality and extremality as before: a CCGF $\psi$ is \emph{minimal} if there is no other CCGF $\psi'$ such that $\psi' \le \psi$, and is \emph{extreme} if it cannot be written as a convex combination of two distinct CCGFs. 
	
	\begin{quote}
	{\bf In the context of CCGFs, we only consider the case $S=b + \Z^n$ with $b \in \R^n\setminus\Z^n$, i.e., when $S$ is an affine-lattice. We will henceforth denote $C_S(R)$ by $C_b(R)$.}
	\end{quote}
	
	It is a well-known fact (see, for example, the recent survey~\cite{basu2015geometric}) that minimal CCGFs are {\em sublinear}, i.e., convex, positively homogeneous functions. Thus, the ``sup" norm is not a well defined norm on this set, and to compare the distance between two distinct minimal CCGFs, one has to define a better metric. We propose the following natural one. For any positively homogenous function $\psi$, we define
	
	\begin{equation}\label{eq:norm-sublinear}\|\psi\| = \sup_{r\in \R^n: \|r\|_1 = 1} \psi(r).\end{equation}
		
One can then obtain the following results about approximating minimal functions using extreme ones.  The following result was suggested to be true in~\cite{espinoza2010computing}.  

\begin{theorem} \label{thm:cont2d}
	Let $n=2$ and let $S=b + \Z^2$ for some $b \in \R^2\setminus\Z^2$. Consider any minimal CCGF $\psi$  for $S$. Then for all $\epsilon > 0$ there exists an extreme function $\psi'$ for $S$ such that $\|\psi - \psi'\| < \epsilon$. 
\end{theorem}

However, we show that this does not hold for $n \ge 3$.

\begin{theorem} \label{thm:contNegative}
	Let $n \ge 3$ and let $S=b + \Z^n$ for some $b \in \R^n\setminus\Z^n$. There exists a minimal function $\psi$ for $S$ and $\epsilon > 0$ such that all extreme functions $\psi'$ for $S$ have $\|\psi - \psi'\| \ge \epsilon$.
\end{theorem}

These results make it plausible that for the infinite group problem the density result may not be true for higher dimensions.


\section{Preliminaries}
\label{sec:prelim}

\subsection{Strongly minimal functions for truncated affine lattices}

The celebrated characterization of strongly minimal inequalities by Gomory and Johnson was recently extended by Cornu\'ejols and Y{\i}ld{\i}z~\cite{yildiz2015integer} to truncated affine lattices 
(actually their result is more general, we only state a special case here).

\begin{theorem}[Cornu\'ejols and Y{\i}ld{\i}z~\cite{yildiz2015integer}] \label{thm:minimal-2} Let $S = b + \Z^n_+$, where $b \in \R^n \setminus \Z^n$ and $b \leq 0$. A continuous function $\pi\colon \R^n \to \R$ is a strongly minimal function if and only if all of the following hold:
\vspace{-4pt}
\begin{itemize}
\item[(i)] $\pi( 0) = 0$, and $\pi(- e_i) = 0$ for all unit vectors $e_i$, $i=1, \ldots, n$,
  \item[(ii)] $\pi$ is subadditive, i.e., $\pi( r_1 +  r_2) \leq \pi( r_1) + \pi( r_2)$, and 
  \item[(iii)]$\pi$ satisfies the symmetry condition, i.e., $\pi( r) + \pi( b- r) = 1$ for all $r \in \R^n$. 
 \end{itemize}
\end{theorem}


	\subsection{Approximations using piecewise linear functions} 
 We say a  function $\phi\colon\R\to \R$ is \emph{piecewise linear} if there is a set of closed nondegenerate intervals $I_j$, $j\in J$ such that $\R = \bigcup_{j\in J}I_j$, any bounded subset of $\R$ intersects only finitely many intervals, and $\phi$ is affine over each interval $I_j$. The endpoints of the intervals $I_j$ are called the {\em breakpoints} of $\phi$. Note that in this definition, a piecewise linear function is continuous. 
 
	The next lemma shows that continuous strongly minimal functions can be approximated by piecewise linear functions that are also strongly minimal; this can be accomplished by restricting the function to a subgroup and performing a linear 
\ifipco
\else
	interpolation.
\fi
 Throughout we use the following notation: Given a subset $X \subseteq \R$ and a function $\phi\colon\R \to \R$, we denote the restriction of $\phi$ to $X$ by $\phi |_X$. 

	\begin{lemma} \label{lemma:interpolation} 
		Let $S$ be an affine lattice $b + \Z$ or a truncated affine lattice $b + \Z_+$ (for $n=1$) with $b\in \Q$. Let $\pi$ be uniformly continuous, strongly minimal function for $S$. Then for every $\epsilon > 0$ there is a continuous strongly minimal function $\pipwl$ for $S$ that is piecewise linear and satisfies $\|\pi - \pipwl\|_\infty \le \epsilon$. 
	\end{lemma}
	
	\begin{proof}
	Since $\pi$ is uniformly continuous and $b$ is rational, there exists a $q \in \Z_+$ such that $b \in \tfrac{1}{q} \Z$ and the  piecewise linear interpolation $\pipwl$ of $\pi|_{\frac{1}{q} \Z}$ satisfies $\|\pi - \pipwl\|_\infty \le \epsilon$.  
	Since $\pi$ is subadditive and symmetric by Theorems~\ref{thm:minimal} and \ref{thm:minimal-2}, $\pi|_{\frac{1}{q}\Z}$ is subadditive and symmetric for all 
rational numbers in $\frac{1}{q} \Z$. It is then easy to see that $\pipwl$ is also subadditive and symmetric and therefore satisfies conditions (ii) and (iii) in Theorems~\ref{thm:minimal} and \ref{thm:minimal-2} (see also~\cite[Theorem 8.3]{basu2016light2}). Also, since all integers are in $\tfrac{1}{q} \Z$, condition (i) in Theorems~\ref{thm:minimal} and \ref{thm:minimal-2} is easily verified for $\pipwl$. Thus, $\pipwl$ is strongly minimal, which concludes the proof. \mqed
	\end{proof}



\subsection{Subadditivity and additivity} We introduce some tools for studying subadditive functions. For any function $\pi\colon \R \to \R$, define a slack function $\Delta\pi\colon \R\times \R \to \R$ as 
\begin{equation}
\Delta\pi(x,y) = \pi(x) + \pi(y) - \pi(x+y).
\end{equation}
Clearly $\pi$ is subadditive if and only if $\Delta \pi \geq 0$.  
We will employ another concept in our analysis which we call the {\em additivity domain}:
$$E(\pi)= \{(x,y) : \Delta \pi(x,y) = 0\}.$$
When $\pi \colon \R \to \R$ is a piecewise linear function periodic modulo $\Z$ with an infinite set of breakpoints  $\B = \{\dots, u_{-1}, u_0, u_1, \dots \}$, by periodicity, we may assume that $\B = \{u_0 , u_1, \dots, u_m \} + \Z$ where $u_0 = 0$, $u_m < 1$ and $u_i < u_{i+1}$. 
The function $\Delta \pi$ is affine on every set $F = \{ (x,y) :  u_i \leq x \leq u_{i+1},  u_j \leq y \leq u_{j+1}, u_k \leq x+y \leq u_{k+1}\}$ where $(u_i, u_{i+1})$, $(u_j, u_{j+1})$, and $(u_k, u_{k+1})$ are pairs of consecutive breakpoints.  The set of all such $F$ forms a polyhedral complex and will be denoted by $\Delta \P_{\B}$. The vertices of any such $F$ will be denoted by $\verts(F)$. If $\mathcal{F}$ is any collection of polyhedra from $\Delta\P_{\B}$, then we define $\verts(\mathcal{F}):= \cup_{F \in \mathcal{F}}\verts(F)$. Note that $\verts(\Delta \P_{\B})$ is exactly the  set of points $(x,y) \in \R^2$ such that either $x,y \in \B$ or $x,x+y \in \B$ or $y, x+y \in \B$.  
The affine structure of $\Delta \pi$ implies the following (for example, see Figure~\ref{fig:pi-delta}).

\begin{lemma}\label{lem:verts}
Let $\pi \colon \R \to \R$ be a piecewise linear function periodic modulo $\Z$ with breakpoints in $\B$. Let $\mathcal{F}$ be a collection of polyhedra from $\Delta\P_{\B}$. If $\Delta \pi(x,y) \geq \gamma$ for all $(x,y) \in \verts(\mathcal{F})$ for some $\gamma > 0$, then $\Delta \pi(x,y) \geq \gamma$ for all $(x,y) \in \cup_{F \in \mathcal{F}} F$. 
In particular, $\Delta \pi(x,y) \geq 0$ for all $(x,y) \in \verts(\Delta \P_{\B})$ if and only if $\pi$ is subadditive.
\end{lemma}  


\subsection{2-Slope Theorems}
Piecewise linear functions where the slope takes exactly two values are referred to as 2-slope functions.
We will use the following two general theorems on extreme functions to show certain 2-slope functions are extreme.

\begin{theorem}[Gomory and Johnson~\cite{infinite}]
\label{thm:2slope}
Let $S = b + \Z$ be an affine lattice and let $\pi$ be a strongly minimal cut generating function for $S$. If $\pi$ is piecewise linear and has exactly two slopes, then it is extreme.
\end{theorem}

Recently, this theorem was extended to the case when $S = b + \Z_+$ using a similar proof as Gomory and Johnson used.

\begin{theorem}[Cornu\'ejols and Y{\i}ld{\i}z~\cite{yildiz2015integer}]
\label{thm:2slopeZP}
Let $S = b + \Z_+$ be a truncated affine lattice and let $\pi \colon \R \to \R$ be a strongly minimal cut generating function for $S$. If $\pi$ is such that $\pi(r) \geq 0$ for all $r \geq 0$ and the restriction of $\pi$ to any compact interval is piecewise linear function with exactly two slopes, then $\pi$ is extreme. 
\end{theorem}

	\subsection{2-Slope fill-in}
Gomory and Johnson~\cite{infinite,infinite2,johnson} described a procedure called the {\em 2-slope fill-in} that allows us to extend subadditive functions from a subgroup of $\R$ to the whole of $\R$. Let $U$ be a subgroup of $\R$. Let $g\colon\R\to\R$ be a sublinear function, \ie, $g$ is subadditive and $g(\lambda r) = \lambda g(r)$ for all $\lambda \geq 0$ and $r\in \R$. The two-slope fill-in of any function $\phi\colon\R\to \R$ with respect to $U$ and $g$ is defined as $$\phi_{\textrm{fill-in}}(r) = \min_{u\in U} \{\phi(u) + g(r-u)\}.$$

\begin{lemma}[Johnson (Section 7 in~\cite{johnson})]
\label{lem:fill-in}
Let $U$ be any subgroup of $\R$ and let $\phi\colon\R\to\R$ be a function such that $\phi	|_U$ is subadditive, i.e., $\phi(u_1 + u_2) \leq \phi(u_1) + \phi(u_2)$ for all $u_1, u_2 \in U$. Suppose $g$ is a sublinear function such that $\phi \leq g$. Then the 2-slope fill-in $\phi_{\textrm{fill-in}}$ of $\phi$ with respect to $U$ and $g$ is subadditive. Moreover, $\phi_{\textrm{fill-in}} \geq \phi$ and $\phi_{\textrm{fill-in}}|_U = \phi|_U$.
\end{lemma}


\vspace{-15pt}
\section{Proof of Theorem~\ref{thm:theoremZ}}
\vspace{-5pt}

	The high-level idea is to apply the 2-slope fill-in procedure to the input function $\bar{\pi}$ and then symmetrize it to produce a 2-slope function $\pi^*$ that satisfies conditions (i), (ii) and (iii) in Theorem~\ref{thm:minimal}, and hence is strongly minimal. Then employing Theorem \ref{thm:2slope} we have that $\pi^*$ is an extreme function. Moreover, we perform the 2-slope fill-in in a way that $\|\bar{\pi} - \pi^*\|_\infty \le \epsilon$, thus giving the desired result. 
	
	The main difficulty in pursuing this line of argument is that the symmetrization step needed after the 2-slope fill-in can easily destroy the desired subadditivity. Therefore, before applying the 2-slope fill-in plus symmetrization we perturb the original function $\bar{\pi}$ to ensure that in most places we have the strict inequality $\pi(x + y) < \pi(x) + \pi(y)$ (and with enough room).
	
	We start describing this perturbation procedure.  For the remainder of this section, we focus on the case where $S = b + \Z$. Also, using periodicity with respect to $\Z$, any function $\pi$ is strongly minimal for $S = b + \Z$ if and only if it is strongly minimal for $S = \bar b + \Z$, where $\bar b \equiv b \pmod 1$.  Hence, without loss of generality, we assume $b \in (0,1)$ throughout this section.


	\vspace{-5pt}	
	\subsection{Equality reducing perturbation}
	
		The perturbation we consider produces a function with equalities (modulo $\Z^2$) only on the border of the unit square and on the symmetry lines $x + y = b$ and $x+y = 1 + b$.  Recall that strongly minimal functions for $S = b + \Z$ are periodic modulo $\Z$ and satisfy the symmetric condition, \ie, $\Delta \pi(x,y) = 0$ whenever $x+y \in b + \Z$. 
		Moreover, only the lines $x+y = b + z$ for $z=0,1$ intersect the cube $[0,1]^2$ since $b \in (0,1)$.
		Define the sets $E_\delta = \{(x,y) : x \in [0,\delta] \cup [1-\delta, 1]\} \cup \{(x,y) : y \in [0,\delta] \cup [1-\delta, 1]\}$ for $\delta > 0$, $E_{b} = \{(x,y) \in [0,1]^2 :  b - \delta \leq x + y \leq b + \delta\}$, and $E_{1+b} = \{(x,y) \in [0,1]^2 : (1+ b) - \delta \leq x + y \leq (1+b) + \delta\}$.
		The main result of this section is the following.
		
		\begin{lemma} \label{lemma:eqReduction}
				Consider a piecewise linear function $\pi$ that is strongly minimal for $b + \Z$. Then for any $\epsilon \in (0,1)$, there is a real number $\delta > 0$ and a function $\picomb$ satisfying the following:
				\vspace{-5pt}
				\begin{enumerate}[(1)]
					\item $\picomb$ is strongly minimal for $b+\Z$.
					\item $\picomb$ is piecewise linear whose breakpoints include $\delta+\Z$ and $-\delta +\Z$. Further, $\picomb$ is linear on $[0,\delta]$ and $[-\delta, 0]$.\label{item:segment}
					\item $\|\pi - \picomb\|_{\infty} \le \epsilon$.
					\item $E(\picomb) \subseteq E_\delta \cup E_{b} \cup E_{1+b}$.
					\item There exists  $\gamma > 0$ such that $\Delta\picomb(x,y) > \gamma$ for all $(x,y) \in [0,1]^2 \setminus \big( E_\delta \cup E_{b} \cup E_{1+b} \big)$.
					\label{item:delta-min}
				\end{enumerate}
		\end{lemma}
		
		The idea behind the proof of this lemma is the observation that if we have a convex combination $\pi = \alpha \pi^1 + (1-\alpha) \pi^2$ with $\alpha \in (0,1)$, then $E(\pi) \subseteq E(\pi^1) \cap E(\pi^2)$. Thus, we will find a function $\hat\pi$ with nice equalities $E(\hat\pi)$ and then set $\picomb$ as roughly $(1 - \epsilon) \pi + \epsilon \hat\pi$.
		The nice function we use is defined for any $\delta \in (0, \min\{\frac{b}{2}, \frac{1-b}{2}\})$ as follows (see Figure~\ref{fig:pi-delta} for an example):

\begin{equation}
\pi_{\delta}(r) = \begin{cases}
\frac{1}{2 \delta}r & r \in [0,\delta] + \Z, \\
\frac{1}{2} & r \in (\delta, b-\delta] + \Z, \\
1 - \frac{1}{2\delta} (b-r) & r \in (b-\delta, b]+ \Z, \\
1 + \frac{1}{2 \delta}(b -r) & r \in (b,b+\delta] + \Z, \\
\frac{1}{2} & r \in (b+\delta,1-\delta] + \Z ,\\
\frac{1}{2} + \frac{1}{2 \delta} (1-\delta-r) & r \in (1-\delta,1] + \Z.
\end{cases}
\end{equation}


%
%



\begin{figure}
\center


\begin{tikzpicture}[scale = 1]
        \draw (-0.2,0)--(5.2,0);
        \draw (0,-0.2)--(0,3.2);
        \draw[black, ultra thick] (0,0)-- (.5,1.5);
        \draw[black,ultra thick] (.5,1.5)--(1.5,1.5);
        \draw[black,ultra thick] (1.5, 1.5) -- (2, 3);
        \draw[black,ultra thick] (2,3) -- (2.5,1.5);
        \draw[black,ultra thick](2.5,1.5) -- (4.5,1.5);
        \draw[black,ultra thick] (4.5,1.5) -- (5,0);
        
        \foreach \x in {0,.5, 1.5,2, 2.5, 4.5, 5} {
            \draw (\x,-0.1)--(\x, 0.1);
        }
        
        \foreach \y in {1.5, 3} {
            \draw (-0.1,\y)--( 0.1, \y);
        }
        
                \node[left] at (0,1.5) {$\frac{1}{2}$};
                     \node[left] at (0,3) {$1$};
        
                        \node[below] at (.5,0) {$\frac{1}{10}$};
                 \node[below] at (1.5,0) {$\frac{3}{10}$};
                  \node[below] at (2,0) {$\frac{2}{5}$};
                   \node[below] at (2.5,0) {$\frac{1}{2}$};
                    \node[below] at (4.5,0) {$\frac{9}{10}$};
                     \node[below] at (5,0) {$1$};

\end{tikzpicture} \ \ \ \ 
\begin{tikzpicture}[scale = 1]

        \pgfmathsetmacro{\a}{5}
        \pgfmathsetmacro{\b}{0}
        \pgfmathsetmacro{\opac}{0.2}
        \pgfmathsetmacro{\opacDark}{0.5}
        \pgfmathsetmacro{\opacGreen}{1}
        
            \draw[opacity = 0](\b,\b) -- (\a,\b) -- (\a,\a) -- (\b,\a) -- cycle;       
            
            \path[fill = black, draw = black, opacity = \opacDark] (1.5,0) --  (2.5,0) -- (0,2.5) -- (0, 1.5) -- cycle;
            \path[fill = black, draw = black, opacity = \opacDark] (1.5,5) --  (2.5,5) -- (5,2.5) -- (5, 1.5) -- cycle;
            
            
            \path[fill = black, opacity = \opac] (0.5,0) -- (0.5,.5) -- (4.5,.5) -- (4.5,0) -- cycle;
            \path[fill = black, opacity = \opac] (0,0) -- (.5,0) -- (.5,5) -- (0,5) -- cycle;
            \path[fill = black,  opacity = \opac] (5,0) -- (4.5,0) -- (4.5,5) -- (5,5) -- cycle;
              \path[fill = black, opacity = \opac] (0.5,5) -- (0.5,4.5) -- (4.5,4.5) -- (4.5,5) -- cycle;
            
            

            \newcommand{\mycolor}{black}   
            
                \draw[draw = \mycolor, opacity = \opacGreen, ultra thick] (2,0) -- (0,2);
            \draw[draw  = \mycolor, opacity = \opacGreen, ultra thick] (2,5) -- (5,2);
            
                \draw[draw = \mycolor, opacity = \opacGreen, ultra thick] (0,0) -- (5,0) -- (5,5) -- (0,5)-- cycle;
            \draw[draw = \mycolor, opacity = \opacGreen, ultra thick] (1.5,.5) -- (.5,1.5);
            \draw[draw  = \mycolor, opacity = \opacGreen, ultra thick] (2.5,4.5) -- (4.5,2.5);
            
            \path[fill = \mycolor, opacity = \opacGreen] (0,0) -- (0.5,0) -- (0,0.5) -- cycle;
            \path[fill = \mycolor, opacity = \opacGreen] (1.5,0) -- (2,0) -- (1.5,.5) -- cycle;
            \path[fill = \mycolor, opacity = \opacGreen] (0,1.5) -- (0,2) -- (.5, 1.5) -- cycle;
            
            \path[fill = \mycolor, opacity = \opacGreen] (2,5) -- (2.5,5) -- (2.5, 4.5) -- cycle;
            \path[fill = \mycolor, opacity = \opacGreen] (5,2) -- (5,2.5) -- (4.5, 2.5) -- cycle;
            \path[fill = \mycolor, opacity = \opacGreen] (5,5) -- (4.5, 5) -- (5,4.5) -- cycle;

        \foreach \x in {0,.5, 1.5,2, 2.5, 4.5, 5} {
            \draw[dashed] (\x,0)--(0,\x);
             \draw[dashed] (\x,0)--(\x,5);
            \draw[dashed] (0,\x)--(5,\x);  
        }
        
            \foreach \x in {0.5, 1.5,2, 2.5, 4.5, 5} {
             \draw[dashed] (\x,5)--(5,\x);
        }

                \node[left] at (0,.5) {$\frac{1}{10}$};
                 \node[left] at (0,1.5) {$\frac{3}{10}$};
                  \node[left] at (0,2) {$\frac{2}{5}$};
                   \node[left] at (0,2.5) {$\frac{1}{2}$};
                    \node[left] at (0,4.5) {$\frac{9}{10}$};
                     \node[left] at (0,5) {$1$};
                     
                        \node[below] at (.5,0) {$\frac{1}{10}$};
                 \node[below] at (1.5,0) {$\frac{3}{10}$};
                  \node[below] at (2,0) {$\frac{2}{5}$};
                   \node[below] at (2.5,0) {$\frac{1}{2}$};
                    \node[below] at (4.5,0) {$\frac{9}{10}$};
                     \node[below] at (5,0) {$1$};         
\end{tikzpicture}

\vspace{-8pt}\caption{On the left is a plot of $\pi_\delta\colon \R \to [0,1]$ for $\delta = \frac{1}{10}$ and $b = \frac{2}{5}$.  This function is periodic modulo $\Z$, so we only display the domain $[0,1]$.  On the right we have drawn the complex $\Delta \P_{\B}$ in dashed lines on the $[0,1]^2$ domain.    The function $\Delta \pi_\delta$ is affine on each cell of $\Delta \P_{\B}$.   The cells of $\Delta \P_{\B}$ filled in black are those contained in the set $E(\pi_\delta)$.  Since $\pi_\delta$ is periodic, $\Delta \pi_\delta$ (and hence $E(\pi_\delta)$) is periodic modulo $\Z^2$.  Covering the set $E(\pi_\delta)$ in gray are the sets $E_\delta, E_{b}$, and $E_{1+b}$.  
  The set $E_\delta$ around the boundary of the box is shaded in lighter gray, while the diagonal strips $E_{b}$ and $E_{1+b}$ are shaded in darker gray.  Notice that $E(\pi_\delta) \subseteq E_\delta \cup E_{b} \cup E_{1+b}$.  In fact, the remaining region $[0,1]^2 \setminus (E_\delta \cup E_{b} \cup E_{1+b})$, colored white,  does not intersect $E(\pi_\delta)$.  Hence $\Delta \pi_\delta > \gamma >0$ on this remaining region.}
%
\remove{
\begin{tikzpicture}[scale = 1]
        \draw (-0.2,0)--(5.2,0);
        \draw (0,-0.2)--(0,3.2);
        \draw[blue, ultra thick] (0,0)-- (.5,1.5);
        \draw[blue,ultra thick] (.5,1.5)--(1.5,1.5);
        \draw[blue,ultra thick] (1.5, 1.5) -- (2, 3);
        \draw[blue,ultra thick] (2,3) -- (2.5,1.5);
        \draw[blue,ultra thick](2.5,1.5) -- (4.5,1.5);
        \draw[blue,ultra thick] (4.5,1.5) -- (5,0);
        
        \foreach \x in {0,.5, 1.5,2, 2.5, 4.5, 5} {
            \draw (\x,-0.1)--(\x, 0.1);
        }
        
        \foreach \y in {1.5, 3} {
            \draw (-0.1,\y)--( 0.1, \y);
        }
        
                \node[left] at (0,1.5) {$\frac{1}{2}$};
                     \node[left] at (0,3) {$1$};
        
                        \node[below] at (.5,0) {$\frac{1}{10}$};
                 \node[below] at (1.5,0) {$\frac{3}{10}$};
                  \node[below] at (2,0) {$\frac{2}{5}$};
                   \node[below] at (2.5,0) {$\frac{1}{2}$};
                    \node[below] at (4.5,0) {$\frac{9}{10}$};
                     \node[below] at (5,0) {$1$};

\end{tikzpicture} \ \ \ \ 
\begin{tikzpicture}[scale = 1]

        \pgfmathsetmacro{\a}{5}
        \pgfmathsetmacro{\b}{0}
        \pgfmathsetmacro{\opac}{0.2}
        \pgfmathsetmacro{\opacGreen}{1}
        
            \draw[opacity = 0](\b,\b) -- (\a,\b) -- (\a,\a) -- (\b,\a) -- cycle;       
            
            \path[fill = red, draw = black, opacity = \opac] (1.5,0) --  (2.5,0) -- (0,2.5) -- (0, 1.5) -- cycle;
            \path[fill = red, draw = black, opacity = \opac] (1.5,5) --  (2.5,5) -- (5,2.5) -- (5, 1.5) -- cycle;
            
            
            \path[fill = blue, opacity = \opac] (0.5,0) -- (0.5,.5) -- (4.5,.5) -- (4.5,0) -- cycle;
            \path[fill = blue, opacity = \opac] (0,0) -- (.5,0) -- (.5,5) -- (0,5) -- cycle;
            \path[fill = blue,  opacity = \opac] (5,0) -- (4.5,0) -- (4.5,5) -- (5,5) -- cycle;
              \path[fill = blue, opacity = \opac] (0.5,5) -- (0.5,4.5) -- (4.5,4.5) -- (4.5,5) -- cycle;
            
            

            \newcommand{\mycolor}{black}   
            
                \draw[draw = \mycolor, opacity = \opacGreen, ultra thick] (2,0) -- (0,2);
            \draw[draw  = \mycolor, opacity = \opacGreen, ultra thick] (2,5) -- (5,2);
            
                \draw[draw = \mycolor, opacity = \opacGreen, ultra thick] (0,0) -- (5,0) -- (5,5) -- (0,5)-- cycle;
            \draw[draw = \mycolor, opacity = \opacGreen, ultra thick] (1.5,.5) -- (.5,1.5);
            \draw[draw  = \mycolor, opacity = \opacGreen, ultra thick] (2.5,4.5) -- (4.5,2.5);
            
            \path[fill = \mycolor, opacity = \opacGreen] (0,0) -- (0.5,0) -- (0,0.5) -- cycle;
            \path[fill = \mycolor, opacity = \opacGreen] (1.5,0) -- (2,0) -- (1.5,.5) -- cycle;
            \path[fill = \mycolor, opacity = \opacGreen] (0,1.5) -- (0,2) -- (.5, 1.5) -- cycle;
            
            \path[fill = \mycolor, opacity = \opacGreen] (2,5) -- (2.5,5) -- (2.5, 4.5) -- cycle;
            \path[fill = \mycolor, opacity = \opacGreen] (5,2) -- (5,2.5) -- (4.5, 2.5) -- cycle;
            \path[fill = \mycolor, opacity = \opacGreen] (5,5) -- (4.5, 5) -- (5,4.5) -- cycle;

        \foreach \x in {0,.5, 1.5,2, 2.5, 4.5, 5} {
            \draw[dashed] (\x,0)--(0,\x);
             \draw[dashed] (\x,0)--(\x,5);
            \draw[dashed] (0,\x)--(5,\x);  
        }
        
            \foreach \x in {0.5, 1.5,2, 2.5, 4.5, 5} {
             \draw[dashed] (\x,5)--(5,\x);
        }

                \node[left] at (0,.5) {$\frac{1}{10}$};
                 \node[left] at (0,1.5) {$\frac{3}{10}$};
                  \node[left] at (0,2) {$\frac{2}{5}$};
                   \node[left] at (0,2.5) {$\frac{1}{2}$};
                    \node[left] at (0,4.5) {$\frac{9}{10}$};
                     \node[left] at (0,5) {$1$};
                     
                        \node[below] at (.5,0) {$\frac{1}{10}$};
                 \node[below] at (1.5,0) {$\frac{3}{10}$};
                  \node[below] at (2,0) {$\frac{2}{5}$};
                   \node[below] at (2.5,0) {$\frac{1}{2}$};
                    \node[below] at (4.5,0) {$\frac{9}{10}$};
                     \node[below] at (5,0) {$1$};         
\end{tikzpicture}

\vspace{-8pt}\caption{On the left is a plot of $\pi_\delta\colon \R \to [0,1]$ for $\delta = \frac{1}{10}$ and $b = \frac{2}{5}$.  This function is periodic modulo $\Z$, so we only display the domain $[0,1]$.  On the right we have drawn the complex $\Delta \P_{\B}$ in dashed lines on the $[0,1]^2$ domain.    The function $\Delta \pi_\delta$ is affine on each cell of $\Delta \P_{\B}$.   The cells of $\Delta \P_{\B}$ filled in black are those contained in the set $E(\pi_\delta)$.  Since $\pi_\delta$ is periodic, $\Delta \pi_\delta$ (and hence $E(\pi_\delta)$) is periodic modulo $\Z^2$.  Covering the set $E(\pi_\delta)$ in a light shading are the sets $E_\delta, E_{b}$, and $E_{1+b}$.  
  The set $E_\delta$ around the boundary of the box is shaded in light blue, while the diagonal strips $E_{b}$ and $E_{1+b}$ are shaded in light red.  Notice that $E(\pi_\delta) \subseteq E_\delta \cup E_{b} \cup E_{1+b}$.  In fact, the remaining region $[0,1]^2 \setminus (E_\delta \cup E_{b} \cup E_{1+b})$, colored white,  does not intersect $E(\pi_\delta)$.  Hence $\Delta \pi_\delta > \gamma >0$ on this remaining region.}
}
\label{fig:pi-delta}
\label{fig:subadditivity-cases} 
\end{figure}

	The following lemma states the key properties of $\pi_\delta$; its proof, presented in 
\ifipco
\else
	Appendix \ref{app:minPiAB},
\fi
 uses the characterization of strong minimality from Theorem \ref{thm:minimal} and requires a case analysis based on the breakpoints of $\pi_\delta$.

	\begin{lemma} \label{lemma:minPiAB}
		For all $\delta \in (0, \min\{\frac{b}{2}, \frac{1-b}{2}\}) $, the function $\pi_{\delta}$ is strongly minimal for $b + \Z$.   
		Furthermore, we have $E(\pi_{\delta}) \subseteq E_{\delta} \cup E_{b} \cup E_{1+b}$ and there exists $\gamma>0$ such that $\Delta \pi_\delta(x,y) > \gamma$ for all $(x,y) \in [0,1]^2 \setminus (E_{\delta} \cup E_{b} \cup E_{1+b})$.
	\end{lemma}



	\paragraph{Proof of Lemma \ref{lemma:eqReduction}.}
		Consider the breakpoints of $\pi$ in the open interval $(0,1)$, let $u_{\min}$ and $u_{\max}$ be respectively the smallest and the largest of these breakpoints. Choose $\delta > 0$ sufficiently small -- more precisely, $\delta < \min\{u_{\min}, 1 - u_{\max},  \frac{b}{2}, \frac{1-b}{2}\}$. 
		
		%
		By Lemma~\ref{lemma:minPiAB}, $\pi_\delta$ is strongly minimal, and $\pi$ is strongly minimal by assumption. Since the conditions (i), (ii) and (iii) in Theorem~\ref{thm:minimal} are maintained under taking convex combinations, the function $\picomb = (1-\epsilon) \pi + \epsilon \pi_\delta$ is also strongly minimal. Thus, condition (1) is satisfied. By the choice of $\delta$ and the fact that $\delta+\Z$ and $-\delta+\Z$ are included in the breakpoints of $\pi_\delta$, condition (2) is also satisfied. Moreover, 
		\begin{align*}
			|\picomb(x) - \pi(x)| &= |(1 - \epsilon) \pi(x) + \epsilon \pi_\delta(x) - \pi(x)| = |-\epsilon \pi(x) + \epsilon \pi_\delta(x)| \le \epsilon,
		\end{align*}
		where the last inequality follows from the fact that $0 \leq \pi(x), \pi_\delta(x) \leq 1$ for all $x$, since both functions are strongly minimal. Thus, condition (3) is satisfied. Finally, by Lemma~\ref{lemma:minPiAB}, there exists a $\hat \gamma > 0$ such that $\Delta\pi_\delta (x,y) > \hat \gamma$ for all $(x,y) \in [0,1]^2 \setminus (E_\delta \cup E_{b} \cup E_{1+b})$.  Since $\Delta \pi \geq 0$, it follows that $\Delta \picomb(x,y) = (1-\epsilon)\Delta\pi(x,y) + \epsilon\Delta\pi_\delta(x,y)\geq \epsilon \Delta \pi_\delta(x,y) > \hat \gamma \epsilon$ for all $(x,y) \in [0,1]^2 \setminus (E_\delta \cup E_{b} \cup E_{1+b})$.  Taking $\gamma = \hat \gamma \epsilon$ completes the proof of conditions (4) and (5). \mqed
	

	\subsection{Symmetric 2-slope fill-in}

	We now show that we can apply the 2-slope fill-in plus a symmetrization procedure to the function $\picomb$ to transform it into a strongly minimal 2-slope function (and hence extreme) while only making small changes to the function values. 
	
	\begin{lemma} \label{lemma:symFillIn}
		Let $\epsilon > 0$ and let $\picomb$ be any function that satisfies the output conditions of Lemma \ref{lemma:eqReduction} (for some $\delta, \gamma >0$) whose breakpoints are rational. There exists a strongly minimal 2-slope piecewise linear function $\pisym$ such that $\|\picomb - \pisym\|_\infty \le \epsilon$. 
	\end{lemma}
	
	\begin{proof}		
		
	By periodicity, we focus on the [0,1] interval. Without loss of generality, we assume that $\epsilon < \tfrac{\gamma}{3}$ where $\gamma$ is given in Lemma~\ref{lemma:eqReduction}\eqref{item:delta-min}. Let $s_+$ and $s_-$ be two slopes of the piecewise linear function $\picomb$ coming from the origin, i.e., let 
$s_+ = \lim_{h \to 0^+} \frac{\picomb(h)}{h}$, $s_- = \lim_{h \to 0^-} \frac{\picomb(h)}{h}$.  Since $\picomb$ is nonnegative,  it follows that $s_+\geq 0$, $s_- \leq 0$. The function $g(r) := \max(s_+ \cdot r, s_-\cdot r),$ is easily seen to be sublinear, and subadditivity of $\picomb$ implies $\picomb \leq g$.

Let $q \in \Z^+$ such that $\tfrac{1}{q} \Z$ such that the breakpoints $\B$ of $\picomb$ and $\frac{b}{2}$ are contained in $\tfrac{1}{q} \Z$ and such that $\frac{1}{q} \max\{s_+, |s_-|\} < \tfrac{\epsilon}{2}$.  Since $\picomb \leq g$, by Lemma~\ref{lem:fill-in}, the fill-in function $\pifill$ of $\picomb$, with respect to $\tfrac{1}{q} \Z$ and $g$, is subadditive.  Unfortunately, $\pifill$ does not necessarily satisfy the symmetry condition and, therefore, is not necessarily a strongly minimal function.  Hence, we further define

$$
\pisym(r) = \begin{cases}
\pifill(r) & r \in [0,\frac{b}{2}]\cup[\frac{1+b}{2},1],\\
1 - \pifill(b-r) & r \in [\frac{b}{2}, \frac{1+b}{2}]
\end{cases}
$$
 
 In the definition of $\pisym$, we have enforced the symmetry condition, possibly sacrificing the subadditivity of the function.  We will show that, given the parameters used in the construction, $\pisym$ is strongly minimal and actually approximates $\picomb$ to the desired precision.

 
By Lemma~\ref{lem:fill-in}, $\pifill \geq \picomb$ and $\picomb(u) = \pifill(u)$ for all $u \in \tfrac{1}{q} \Z$.   Since $\picomb$ is period modulo $\Z$,  the function $\pifill$ inherits this property. Moreover, restricted to $[0,1]$, $\pifill$ is the pointwise minimum of a finite collection of piecewise linear functions and therefore, $\pifill$ is also piecewise linear.  Furthermore, the maximum slope in absolute value of $\pifill$ is $s = \max\{ s_+, |s_-|\}$.  Therefore $s$ is also a bound on the maximum slope in absolute value for $\picomb$.  Hence, 
 $$
 |\pifill(r) - \picomb(r)| \leq |\pifill(u) - \picomb(u)| + 2 s |u-r| \leq \epsilon,
  $$
 where $u \in \tfrac{1}{q} \Z$ is the closest point in $\frac1q\Z$ to $r$.  
Thus, we have established that $\| \pifill - \picomb\|_\infty \leq \epsilon$.  
 Observe that $|\pisym(r) - \picomb(r)|=|\pifill(r) - \picomb(r)| $ for all $r \in [0,\frac{b}{2}]\cup[\frac{1+b}{2},1]$, and $|\pisym(r) - \picomb(r)| = |\pifill(b-r) - \picomb(b-r)|$ for all $r \in [\frac{b}{2}, \frac{1+b}{2}]$ because of the symmetry of $\picomb$.  Therefore we also have that $\| \pisym - \picomb\|_\infty \leq \epsilon$.
 
 Next, observe that $\pisym$ has the same slopes as $\pifill$, and therefore is a 2-slope piecewise linear function.
 
 Finally, we establish that $\pisym$ is a strongly minimal function.   Since it is clear that $\pisym(0) = 0$, and $\pisym$ satisfies the symmetry condition, by Theorem~\ref{thm:minimal}, we only need to show that $\pisym$ satisfies the subadditivity condition $\pisym(x+y) \le \pisym(x) + \pisym(y)$; equivalently, $\Delta\pisym(x,y) \geq 0$. 
 This is established by the following case analysis for each $(x,y) \in [0,1]^2$.
 
%
 
 \textbf{Case 1.} 
Suppose  $(x,y) \in [0,1]^2 \setminus \big( E_\delta \cup E_{b} \cup E_{1+b}\big)$.\\
By Lemma~\ref{lemma:eqReduction}\eqref{item:delta-min}, $\Delta \picomb(x,y) > \gamma > 0$.  Since $\|\pisym - \picomb\|_\infty \leq \epsilon$, it follows that 
$$
\Delta \pisym(x,y) \geq \Delta\picomb(x,y) - 3 \|\pisym - \picomb\|_\infty \geq 0.
$$

\textbf{Case 2.}  
  Suppose $(x,y) \in E_\delta$.\\
By Lemma~\ref{lemma:eqReduction}\eqref{item:segment}, the slope of $\picomb$ on the interval $[0,\delta]$ is $s_+$, while the slope on the interval $[1-\delta, 1]$ is $s_-$.  We claim that $\picomb = \pifill$ on the intervals $[0,\delta]$ and $[1-\delta, \delta]$.  To see this, consider any two consecutive points $u_1, u_2 \in [0,\delta] \cap \tfrac{1}{q} \Z$.  
For any $r \in [u_1, u_2]$, we have 
\begin{align*}
\picomb(r) &= \picomb(u_1) + s_+ (r - u_1) = \picomb(u_1) + g(r - u_1) \\
&\geq \min_{ u \in U} \picomb(u) + g(r - u) = \pifill(r) \geq \picomb(r).
\end{align*}
The first equality comes from the second part of Lemma~\ref{lemma:eqReduction}\eqref{item:segment}. The last inequality comes from Lemma~\ref{lem:fill-in}. Since this holds for any points $u_1, u_2 \in \tfrac{1}{q} \Z \cap [0,\delta]$ and $0, \delta \in \tfrac{1}{q}\Z$ by Lemma~\ref{lemma:eqReduction}\eqref{item:segment} and the assumption that all breakpoints of $\picomb$ lie in $\frac1q\Z$, the claim holds on the interval $[0,\delta]$. A similar argument verifies the claim on the interval $[1-\delta ,1]$ by showing that $\pifill$ takes slope $s_-$ on this interval.

 Therefore, for $x \in [0,\delta]$ we have
 $$\pisym(x) =  s_+ \cdot x \geq  s_+ \cdot \alpha_1 +  s_- \cdot \alpha_2 = \pi(x+y) - \pi(y),$$
 where $\alpha_1$ and $\alpha_2$ are the lengths of the subsets of the interval $[y, x+y]$ taking slopes $s_+$ and $s_-$ respectively.  The inequality holds since $\alpha_1 + \alpha_2 = x$ and $s_+ \geq s_-$.
 
 On the other hand, if $x \in [1-\delta, 1]$,  then
 \begin{align*}
- \pisym(x) &= -\pisym(x-1) = s_- \cdot (1-x)\\
 &\leq s_+\cdot \alpha_1 + s_-\cdot \alpha_2 = \pisym(y) - \pisym(x+y-1) =  \pisym(y)-\pisym(x+y),
 \end{align*}
 where $\alpha_1, \alpha_2$ are the lengths of the subsets of the interval $[x+y-1,y]$ taking slopes $s_+$ and $s_-$ respectively.  The inequality holds since $\alpha_1 + \alpha_2 = 1-x$ and $s_+ \geq s_-$. Here we used the fact that $\pisym$ is periodic modulo $\Z$.  

\textbf{Case 3.} 
Suppose  $(x,y)  \in E_{b} \cup E_{1+b}$.\\
Suppose first that $(x,y) \in E_{b}$.  Then $x + y = b - \beta$ for some $\beta \in [-\delta, \delta]$.
By Case 2, it follows that $\pisym(\beta) + \pisym(x) \geq \pisym(x + \beta)$.  Therefore, $-( \pisym(\beta) + \pisym(x) ) \leq - \pisym(x+ \beta)$.
Since $\pisym$ satisfies the symmetry condition, we have
\begin{align*}
\pisym(b - \beta) &= 1 - \pisym(\beta) 
= 1 - \pisym(\beta) - \pisym(x) + \pisym(x)\\
& \leq 1 - \pisym(x + \beta) + \pisym(x) 
= \pisym(b - x - \beta) + \pisym(x) \\
&= \pisym(x) + \pisym(y).
\end{align*}
 The proof is similar for $(x,y) \in E_{1+b}$.   

Since Cases 1-3 cover all options for $(x,y) \in [0,1]^2$, we see that $\pisym$ is indeed subadditive.  This concludes the proof.
%
%
%
	\mqed
	\end{proof}


	\subsection{Concluding the proof of Theorem \ref{thm:theoremZ}}
	
	
	Consider a strongly minimal function $\bar{\pi}$ for $S = b + \Z$. Since $\bar\pi$ is continuous and periodic with period 1, it is actually uniformly continuous. Thus, we can apply Lemma \ref{lemma:interpolation} to obtain a piecewise linear function $\pipwl$ that is strongly minimal for $S$ and satisfies $\|\bar{\pi} - \pipwl\|_\infty \le \frac{\epsilon}{3}$. Then we employ the equality reduction Lemma \ref{lemma:eqReduction} over $\pipwl$ to obtain a strongly minimal function $\picomb$ with $\|\pipwl - \picomb\|_\infty \le \frac{\epsilon}{3}$. Then we can apply Lemma \ref{lemma:symFillIn} to $\picomb$ to obtain a function $\pisym$ with $\|\picomb - \pisym\|_\infty \le \frac{\epsilon}{3}$ satisfying the other properties given by the lemma. Then the 2-Slope Theorem \ref{thm:2slope} implies that $\pisym$ is extreme, and triangle inequality gives $\|\bar{\pi} - \pisym\|_\infty \le \|\bar{\pi} - \pipwl\|_\infty + \|\pipwl - \picomb\|_\infty + \|\picomb - \pisym\|_\infty \le \epsilon$. This concludes the proof.
	
	
	\section{Proof of Theorem \ref{thm:theoremZPlus}}

	
	The high-level idea of the proof of Theorem \ref{thm:theoremZPlus} is to take the input function $\bar\pi$, which is strongly minimal and quasi-periodic, and remove a linear term from it and scale the domain to obtain a function $\tilde{\pi}$ that is periodic modulo $\Z$ (and in fact strongly minimal). Then we can apply Theorem \ref{thm:theoremZ} to this transformed function $\tilde{\pi}$ to obtain an extreme function $\tpisym$ close to it and then undo the transformation over $\tpisym$ to obtain an extreme function $\pi^*$. The only caveat is that in this last step simply undoing the function transformation does not give us an extreme function: an extra correction step needs to take place to correct the fact that $\tpisym$ is a (slight) modification of $\tilde{\pi}$.
		
		One can transform a quasi-periodic function into a periodic one by removing a linear term (the proof can be readily verified). 
		
\begin{lemma}\label{lem:sum}
Let $\phi$ be quasi-periodic with period $d$, and let $c \in \R$ be such that $\phi(x + d) = \phi(x) + c$. Then the function $\tilde \phi(x) := \phi(x) - \frac{c}{d}x$ is periodic with period $d$.
\end{lemma}

%

	We also need the following lemma which follows from~\cite[Theorem 7.5.1]{hillePhillips}.

\begin{lemma}\label{lem:sub-per-pos}
Let $\phi\colon\R\to\R$ be a continuous, subadditive, periodic function and $\phi(0) = 0$. Then $\phi \geq 0$.
\end{lemma}


	We proceed with proving Theorem \ref{thm:theoremZPlus}. Consider a truncated affine lattice $S = b + \Z_+$ with $b \in \Q\setminus \Z$ and $b \leq 0$. Consider a continuous, strongly minimal, quasi-periodic function $\bar\pi$ for $S$ with rational period $d$ and let $c \in \R$ be such that $\bar\pi(x + d) = \bar\pi(x) + c$. The rationality of $d$, combined with~\cite[Theorem 7]{yildiz2015integer}, implies that $c \geq 0$. Thus, by Lemma~\ref{lem:sum}, there exists $\alpha \geq 0$ such that the function $\bar\pi(r) - \alpha\cdot r$ is periodic with period $d$. Define the transformed function $\tilde{\pi} \colon \R \rightarrow \R$ by removing a linear term and scaling the domain as 
	\vspace{-3pt}
	\begin{align*}
		\tilde{\pi}(r) = \frac{1}{1 - \alpha b} \cdot \left(\bar\pi(d \cdot r) - \alpha (d \cdot r)\right).
	\end{align*}
	
	Observe that $1-\alpha b \geq 0$ since $\alpha \geq 0$ and $b \leq 0$. Not only is $\tilde\pi$ periodic with period 1, it is in fact strongly minimal for an appropriately transformed set 
\ifipco
\else
	$\tilde{S}$.
\fi

	\begin{lemma} \label{lemma:SM}
		The function $\tilde{\pi}$ is strongly minimal for $\tilde{S} = \frac{b}{d} + \Z$.
	\end{lemma}
	
	\begin{proof}
		Since $\bar\pi$ is quasi-periodic with period $d$, subadditive, and has $\bar\pi(0) = 0$ 
	it is easy to check that $\tilde{\pi}$ is periodic with period 1, subadditive, and has $\tilde{\pi}(0) = 0$. Moreover, the $b$-symmetry of $\bar\pi$ implies that $\tilde{\pi}$ is $\frac{b}{d}$-symmetric:
	\begin{align*}
		\tilde{\pi}(r) + \tilde{\pi}\left(\frac{b}{d}-r\right) &= \frac{1}{1 - \alpha b} \cdot \left(\bar\pi(d \cdot r) - \alpha (d \cdot r) + \bar\pi(b-d \cdot r) - \alpha (b-d \cdot r)  \right)\\
			&= \frac{1}{1 - \alpha b} \cdot \left( \bar\pi(b) - \alpha b \right) = 1.
	\end{align*}
	
	Since $\bar\pi$ is continuous, $\tilde \pi$ is continuous and so, by Lemma~\ref{lem:sub-per-pos}, we have that $\tilde{\pi}$ is nonnegative. Thus, conditions (i), (ii), (iii) in Theorem~\ref{thm:minimal} are all satisfied, which gives the desired result. \mqed
	\end{proof}
		
	Recall the parameters $M$ and $\epsilon>0$ in the statement of Theorem~\ref{thm:theoremZPlus}. Now set $\epsilon' > 0$ small enough so that $1 +\epsilon'b \geq \frac12$ and $$\max\bigg\{\left(\frac{1}{1+\epsilon' b} - 1\right),\left(1-\frac{1}{1 - \epsilon' b}\right) \bigg\}\cdot \max_{y \in [-M,M]} |\bar\pi(y)| + \epsilon' (M + 2) \le \epsilon.$$ Apply Theorem \ref{thm:theoremZ}, 
	 with approximation parameter $\frac{\epsilon'}{1 -\alpha b}$, to obtain a 2-slope function $\tpisym$ that is strongly minimal for $\tilde{S}=\frac{b}{d}+\Z$ and has $\|\tilde{\pi} - \tpisym\|_\infty \le \frac{\epsilon'}{1 - \alpha b}$.

	
	We undo the transformation over $\tpisym$ by rescaling the domain and function values, and adding back the linear term to define $\pi' \colon \R \rightarrow \R$ by setting $$\pi'(r) = (1 - \alpha b) \cdot \tpisym\left(\frac{r}{d}\right) + \alpha r.$$ Again notice that $\pi'$ satisfies quasi-periodicity (with period $d$), subadditivity, and $\pi'(0) = 0$. Also, since $\tpisym$ is symmetric about $\frac{b}{d}$, we obtain that $\pi'$ is symmetric about $b$:
	\begin{align*}
		\pi'(r) + \pi'(b-r) &= (1 - \alpha b) \cdot \left( \tpisym\left(\frac{r}{d}\right) + \tpisym\left(\frac{b-r}{d}\right)  \right) + \alpha r + \alpha (b-r) = 1.
	\end{align*}
	In addition, $\|\bar\pi - \pi'\|_\infty \le \epsilon'$:
	\begin{align*}
		|\pi'(r) - \bar\pi(r)| &= \left\vert(1-\alpha b) \cdot \tpisym\left(\frac{r}{d}\right) + \alpha r - \bar\pi(r)\right\vert \\
		& = \left\vert (1-\alpha b) \left(\tpisym\left(\frac{r}{d}\right) - \tilde{\pi}\left(\frac{r}{d}\right)\right) + (1 -\alpha b) \tilde{\pi}\left(\frac{r}{d}\right) + \alpha r - \bar\pi(r) \right\vert\\
		& \le (1-\alpha b) \|\tpisym - \tilde{\pi}\|_{\infty} \le \epsilon'.
	\end{align*}
	
	Thus the function $\pi'$ satisfies all conditions in Theorem~\ref{thm:minimal-2}, except that $\pi'(-1)$ may be different from $0$, and thus it may not be strongly minimal (and hence extreme). However, we can correct this in the following way.
	
	 Let $\beta = \pi'(b) + \pi'(-1) \cdot b = 1 +\pi'(-1) \cdot b$. Since $|\pi'(-1)| \leq\epsilon'$ and by choice of $\epsilon'$ we have $1 + \epsilon' b > 0$, we obtain $\beta > 0$. Define $\pi^*(r) = \frac{1}{\beta} (\pi'(r) + \pi'(-1) \cdot r)$.  
\ifipco
\else
	We show that this function $\pi^*$ now satisfies the conditions of Theorem~\ref{thm:minimal-2}, and thus is strongly minimal, and is close to $\bar{\pi}$.
\fi
 
	 \begin{lemma}
	 \label{lem:pi-star}
	 The function $\pi^*$ is a piecewise linear 2-slope function that is strongly minimal for $S = b + \Z_+$.  Furthermore, $|\bar\pi(r) - \pi^*(r)| \leq \epsilon$ for all $r \in [-M,M]$.
	 \end{lemma}
	
		\begin{proof}
				 Observe that $\pi^*(0) = 0$ and $\pi^*(-1) = 0$. Since $\pi'$ is subadditive, piecewise linear and 2-slope, and $\pi^*$ is obtained from $\pi'$ by adding a linear term and scaling by a positive constant, we observe that $\pi^*$ is subadditive, piecewise linear, and 2-slope. 
		 Consider $\pi^*(r) + \pi^*(b-r) = \frac{1}{\beta}[\pi'(r) + \pi'(b-r) + \pi'(-1)(r) + \pi'(-1)(b-r)] = \frac{1}{\beta}[1+\pi'(-1)b] = 1$, confirming that $\pi^*$ satisfies the symmetry condition. Thus, by Theorem~\ref{thm:minimal-2}, $\pi^*$ is strongly minimal for $S=b+\Z_+$.
		 
		 Now we show that $|\pi^*(r) - \pi'(r)| \le \epsilon$ for all $r \in [-M, M]$. To start, notice
		%
		\begin{align*}
			|\pi^*(r) - \pi'(r)| = \left\vert \left(\frac{1}{\beta} - 1\right)\pi'(r) + \pi'(-1) \cdot r \right\vert \le \left|\left(1 - \frac{1}{\beta}\right)\right| |\pi'(r)| + |\pi'(-1)| \cdot |r|.
		\end{align*}
		Since $|\pi'(-1)| \le \epsilon'$, if $\pi'(-1) > 0$ then $\beta \geq 1 + \epsilon' b$ and so $\left|\left(1 - \frac{1}{\beta}\right)\right| = \frac{1}{\beta} - 1 \leq  \frac{1}{1+\epsilon' b} - 1$, and if $\pi'(-1) \leq 0$ then $\beta \leq 1 - \epsilon' b$ and therefore $\left|\left(1 - \frac{1}{\beta}\right)\right| = 1 - \frac{1}{\beta} \leq 1 -  \frac{1}{1-\epsilon' b}$. By plugging these bounds in the previous displayed inequality and taking a supremum on the right-hand side over all $r \in [-M, M]$, we see that
		for all $r \in [-M, M]$ 
%
%
%
	 
	%
	\begin{align*}
		|\pi^*(r) - \pi'(r)| &\le \max\bigg\{\left(\frac{1}{1+\epsilon' b} - 1\right),\left(1-\frac{1}{1 - \epsilon' b}\right) \bigg\}\cdot\max_{y \in [-M,M]} |\pi'(y)| + \epsilon' M \\
			&\le \max\bigg\{\left(\frac{1}{1+\epsilon' b} - 1\right),\left(1-\frac{1}{1 - \epsilon' b}\right) \bigg\}\cdot\max_{y \in [-M,M]} |\bar\pi(y)| + \epsilon' (M + 1) \\
			&\le \epsilon - \epsilon'.
	\end{align*}
	The second inequality follows because $||\pi'-\bar \pi||_\infty \leq \epsilon'$, and $\frac{1}{1+\epsilon' b} - 1$ and $1-\frac{1}{1 - \epsilon' b}$ are both in $(0,1)$ because $1 + \epsilon' b \geq \frac12$ and $b \leq 0$.
	
	Finally, from the triangle inequality we get $|\pi^*(r) - \bar\pi(r)| \le |\pi^*(r) - \pi'(r)| + |\pi'(r) - \bar\pi(r)| \le \epsilon$ for all $r \in [-M, M]$. This concludes the proof. \mqed
		\end{proof}
	
	Finally, from Theorem \ref{thm:2slopeZP} we have that $\pi^*$ is extreme. This concludes the proof of Theorem \ref{thm:theoremZPlus}. 
	

\section{Continuous Infinite Relaxation}

	For any convex set $K$, we use $\intr(K)$, $\relint(K)$ and $\lin(K)$ to denote the interior, relative interior and the linearity space of $K$, respectively. There is a well-known connection between CCGFs for $S = b+\Z^n$ and \emph{$S$-free convex sets}. A set $K \subseteq \R^n$ is \emph{$S$-free} if it does not contain any points from $S$ in its interior. A \emph{maximal} $S$-free convex set is one that is inclusion-wise maximal. 
For a closed convex set $K \subseteq \R^n$ with $0 \in \intr(K)$, define the function 
\begin{equation}
\gamma_{K}(r) = \inf\{t > 0 : r \in tK \text{ for all } r \in \R^n\}
\end{equation}
(this is known as the \emph{gauge} function of the set $K $). Since the origin is in the interior of $K$, $\gamma_{K}(r)   < + \infty$ for all $r \in \R^n$. The following theorem states the connection between minimal valid functions and lattice-free sets, and collects other important properties of the latter. Recall the norm defined on positively homogeneous functions by~\eqref{eq:norm-sublinear}.


\begin{theorem}[Theorems 4.4,4.5,4.9 in~\cite{basu2015geometric}] \label{thm:geoCGF} Let $S = b+\Z^n$ for some $b \in \R^n\setminus\Z^n$.
\begin{enumerate}
\item A function $\psi \colon \R^n \to \R$ is a minimal CCGF for $S$ if and only if there exists some maximal $S$-free convex set $K \subseteq \R^n$ such that $0 \in \intr(K)$ and $\psi = \gamma_{K}$.
\item A full dimensional convex set $K \subseteq \R^n$ is a maximal $S$-free convex set if and only if it is an $S$-free polyhedron such that each facet contains a point of $S$ in its relative interior.

\item If a maximal $S$-free polyhedron $K$ with $0 \in \intr(K)$ is given by $K = \{x \in \R^n : a_i \cdot x \leq 1, \forall i \in I\}$ for some finite set $I$, then 
$$\gamma_{K}(r) = \max_{i \in I} a_i \cdot r,$$
and \begin{equation}\label{eq:formula-norm}\|\gamma_K\| = \max_{i\in I} |a_i|_\infty.\end{equation}

\end{enumerate}
\end{theorem}

	The next theorem, in particular, connects the extremality of a CCGF with that of a specific problem $C_b(r^1, \ldots, r^k)$.

\begin{theorem}[Theorem 1.5 in~\cite{bccz}]
\label{thm:corner-rays}
Let $S = b+\Z^n$ for some $b \in \R^n\setminus\Z^n$. Let $K$ be a maximal $S$-free convex set in $\R^n$ with $0$ in its interior. Let $L = \lin(K)$ and let $P = K \cap L^\perp$. Then $K = P + L$, $L$ is a rational space, and $P$ is a polytope. Moreover, let $r^1,\dots , r^k$ be the vertices of $P$, and $r^{k+1},\dots, r^{k+h}$ be a rational basis of $L$.  
Then the function $\gamma_{K}$ is extreme for $C_f$ if and only if the inequality $\sum_{i=1}^k s_i \geq 1$ is extreme for the polyhedron $\conv(C_b(r^1, \dots, r^{k+h}))$. 
\end{theorem}


\subsection{Proof of Theorem \ref{thm:cont2d}}

	We now prove Theorem \ref{thm:cont2d}, which asserts that in the two-dimensional case $n = 2$, extreme CCGFs are dense in the set of minimal functions. For that we need the following characterization of extreme functions in this two-dimensional case. Recall that a \emph{split set} in $\R^2$ is one of the form $\{(x_1, x_2) : \pi_0 \le \pi_1 x_1 + \pi_2 x_2 \le \pi_0 + 1\}$ for some $\pi_0, \pi_1, \pi_2 \in \R$.

\begin{theorem}[Theorems 3.8, 3.10, 4.1 \cite{cm}]
\label{thm:2-d-char}
	In $\R^2$, if $S = b+\Z^2$ for some $b \in \R^2\setminus\Z^2$, then the full-dimensional maximal $S$-free polyhedra are splits, triangles, and quadrilaterals. Moreover, consider a maximal $S$-free convex set $K \subseteq \R^2$ with $0 \in \intr(K)$. Then:
\begin{enumerate}
\item 
If $K$ is a split set, then $\gamma_{K}$ is an extreme CCGF for $S$.  
\item 
If $K$ is a triangle with vertices $r^1, r^2, r^3$, then $\gamma_{K}$ defines a facet of $C_b(r^1, r^2, r^3)$, i.e., $\sum_{i=1}^3 \gamma_K(r^i)s_i \geq 1$ is a facet of $C_b(r^1, r^2, r^3)$. Note further that $\gamma_K(r^i) = 1$ for all $i=1, 2$ and $3$.
\item 
Suppose $K$ is a quadrilateral. Let its vertices be $r^1, \ldots, r^4$, and let $w^i\in S$ be a point on the edge $[r^i, r^{i+1}]$ (indices taken modulo $4$).
Then $\gamma_{K}$ defines a facet of $C_b(r^1,r^2,r^3,r^4)$ if and only if there is no $t \in \R_+$ such that the point $w^i$ divides the edge joining $r^i$ to $r^{i+1}$ in a ratio $t$ for odd $i$ and in a ratio $1/t$ for even $i$, i.e.
\begin{equation}\label{eq:ratio-cond}
\frac{\|w^i - r^i \|}{\|w^i - r^{i+1}\|}  = \begin{cases} t \text{ for } i=1,3, \\ \frac{1}{t} \text{ for } i = 2,4.  \end{cases}
\end{equation}
\end{enumerate}
\end{theorem}


\paragraph{Proof of Theorem \ref{thm:cont2d}.}
By Theorem \ref{thm:geoCGF},  $\psi = \gamma_{K}$ for a maximal $S$-free convex set $K = \{ x \in \R^2 : a_i x \leq 1\}$ with $0 \in \intr(K)$.  If $\gamma_K$ is extreme, then there is nothing to be done. Therefore, by Theorems~\ref{thm:corner-rays} and~\ref{thm:2-d-char}, we may assume that $K$ is a quadrilateral such that there exists a $t\in \R_+$ satisfying~\ref{eq:ratio-cond}. Let $w_1 \in S$ be such that $a_1\cdot w_1 = 1$. Let $a$ be any vector in $\R^2$ with $|a|_\infty = 1$ that is orthogonal to $w_1$, i.e., $a\cdot w_1 = 0$. Now, define $\tilde a^1 = a^1 + \epsilon a$ and $\tilde a^i = a^i$ for $i=2,3$ and $4$. One can show that for small enough $\epsilon$, $\tilde K = \{ x \in \R^2 : \tilde a_i x \leq 1\}$ is a maximal $S$-free quadrilateral and there is no $t\in \R_+$ such that~\eqref{eq:ratio-cond} holds, and so $\gamma_{\tilde K}$ is extreme by Theorems~\ref{thm:corner-rays} and~\ref{thm:2-d-char}. Moreover, by~\eqref{eq:formula-norm}, $\|\gamma_K\| = \max_i |a_i|_\infty$ and $\gamma_{\tilde K} = \max_i |\tilde a_i|_\infty$. Therefore, $\| \gamma_K - \gamma_{\tilde K} \| \leq \epsilon$.
\mqed


	\subsection{Proof of Theorem \ref{thm:contNegative}}

	The following rank identity is a direct consequence of equation (4.5.1) of \cite{meyer} (note that if a matrix $A$ has full column rank then it has a 0-dimensional kernel).

\begin{lemma} 
\label{lem:linear-algebra}
Let $A$ be an $m \times d$ matrix and $B$ an $d \times p$ matrix. If $\rank(A) = d$, then $\rank(A\cdot B) = \rank(B)$. 
\end{lemma}

The proof of the following theorem follows the proof of Theorem 3.8 in \cite{cm} for the 2-dimensional case.
\begin{theorem}
\label{thm:extreme-simplices}
	Let $S = b+\Z^n$ for some $b \in \R^n\setminus\Z^n$. Let $K\subseteq \R^n$ be a maximal $S$-free simplex with $0 \in \intr(K)$ and with exactly one point from $S$ on each facet. Then $\gamma_{K}$ is an extreme CCGF if and only if the affine hull of $K\cap S$ is all of $\R^n$.  
\end{theorem}

\begin{proof}
Let $y^1, \dots, y^{n+1} \in S$ be the points from $S$ on the boundary of $K$ and let $r^1, \dots, r^{n+1} \in \R^n$ be the vertices of $K$.  
By Theorem~\ref{thm:corner-rays}, $\psi$ is an extreme CCGF if and only if $\sum_{i=1}^{n+1} s_i \geq 1$ is extreme for $C_b(r^1, \dots, r^{n+1})$.  

 Let $F$ be the face of $C_b(r^1,\dots ,r^{n+1})$ defined by $\sum_{i=1}^{n+1} s_i = 1$. 
Since $C_b(r^1, \dots, r^{n+1})$ is full dimensional (see \cite{cm} Lemma~1.6), it follows that $F$ is a facet if and only if there exist $n+1$ affinely independent vectors $s^j$, for $j=1, \dots, n+1$, such that 
 $$
 \sum_{i=1}^{n+1} s_i^j = 1  \ \ \text{ and } z^j = 
 \sum_{i=1}^{n+1} r^i s_i^j \ \text{ is in }S
 $$
 Since $z^j$ is a convex combination of the vertices of $K$, $z^j \in K$ also. Thus, in fact, each $z^j \in \{y^1, \ldots, y^{n+1}\}$. Let $Z, R, X$ be the matrix with columns $\{z^j\}_{j=1, \ldots, n+1}$, $\{r^i\}_{i=1, \ldots, n+1}$, and $\{s^j\}_{j=1, \ldots, n+1}$, respectively. Let $\bar Z$ and $\bar R$ denote the matrices $Z$ and $R$, respectively, with a row of 1s added as the last row of the matrix. Then it follows that
 $$
\bar Z= \bar R \cdot X.
$$

Since $K$ is a full dimensional simplex, $\rank(\bar R) = n+1$. Therefore, by Lemma~\ref{lem:linear-algebra},  $\rank(X) = \rank(\bar Z)$.  Hence, the vectors $s^i$ for $i=1, \dots, n+1$ are affinely independent if and only if the points $z^1, \dots, z^{n+1}$ are affinely independent, which can happen if and only if $\{z^1, \dots, z^{n+1}\} = \{y^1, \dots, y^{n+1}\}$ and $y^1, \dots, y^{n+1}$ are affinely independent. \mqed
\end{proof}

In a similar way, it is easy to see that all $S$-free polyhedra $K$ whose vertices are in $S$ give rise to extreme CCGFs. This is simply because the vector $s^j = e^j$ is valid (where $e^j$ is the $j$-th unit vector in $\R^k$ where $k$ is the number of vertices of $K$) since every corner ray points to a point in $S$.  Thus, the $X$ matrix has full rank.

\begin{lemma}\label{ex:simplex}  For every $n\geq 3$ there exists a maximal $\Z^n$-free simplex in $\R^n$ such that the affine hull of the integer points on its boundary is a $(n-1)$-dimensional subspace of $\R^n$.
\end{lemma}
\begin{proof}
Let $0< \epsilon < \tfrac{1}{4}$. 
For $n\geq 3$, let $A, \bar A \in \R^{(n-1) \times n}$, $C \in \R^{2 \times n}$, $b \in \R^{n-1}$ and $d \in \R^2$ be given as
$$
A = \left[\begin{array}{ccccccccc} -1 & \epsilon & 0 & \dots &  0&  \\
  & \ddots  & \ddots   & \ddots & \vdots\\
  & &-1 & \epsilon  &     0 \\
  \epsilon &  &  & -1& 0 
 \end{array}
 \right],
 \ \ 
 \bar A = \left[\begin{array}{ccccccccc}  &  &  && \\
  & & &  &  \\ \\
  & && & &   & \\
  0 & -\bar\epsilon & 0 &\cdots& 0& &  
 \end{array}
 \right],
\  \
b = \left[\begin{array}{c}\epsilon \\ \vdots \\ \epsilon   \end{array}\right], 
$$
$$
C = \left[\begin{array}{cc|ccc|c} 
 - 1 & -1& - \frac{3}{2}     & \cdots  &-\frac{3}{2}  & -\frac{1}{\epsilon}\\
 1 & 1 & \frac{3}{2}  & \cdots & \frac{3}{2} & \tfrac{1}{2\epsilon} 
 \end{array}
 \right], \ \ \text{ and }
 d = \left[\begin{array}{c}0 \\ 2 \end{array}\right]. 
$$
Note that for $n=3$, $$C = \left[\begin{array}{ccc} 
 - 1 & - 1 & -\frac{1}{\epsilon}\\
 1 & 1  & \tfrac{1}{2\epsilon} 
 \end{array}
 \right],$$
where $\bar \epsilon>0$ is chosen to be arbitrarily small.

Let $\Delta_3 = \bar \Delta_3 = \{ x\in \R^3 : Ax \leq b, Cx \leq d\}$.  For all $n \geq 4$, let $\Delta_n = \{ x \in \R^n : Ax \leq b, C x \leq d\}$ and $\bar\Delta_n = \{ x \in \R^n : (A+ \bar A)x \leq b, C x \leq d\}$.   We will show that $\bar \Delta_n$ is a maximal $\Z^n$-free simplex containing exactly $n+1$ lattice points that are contained in the $(n-1)$-dimensional subspace $\{ x: x_{n} = 0\}$ of $\R^n$.  We will only show the calculations for $n\geq 4$ since the calculations for $n=3$ are similar.  

To do so, we prove first that $\Delta_n$ is a $\Z^n$-free simplex with containing exactly $n+1$ lattice points that are contained in the $(n-1)$-dimensional subspace $\{ x: x_{n} = 0\}$ of $\R^n$.  Then, since $\bar \Delta_n$ comes from perturbing the inequality matrix from $\Delta_n$, and $\bar \Delta_n$ contains all its lattice points on the relative interior of its facets, that is, one for each facet, we will have proved that it is a maximal $\Z^n$-free convex set.

Let $W_i = [\frac{1}{\epsilon^{n-i}}, \frac{1}{\epsilon^{n-i-1}}, \dots, \frac{1}{\epsilon}, \frac{1}{\epsilon^{n-1}}, \dots, \frac{1}{\epsilon^{n-i+1}}]^T$. 
For example, $W_1 = [\frac{1}{\epsilon^{n-1}}, \frac{1}{\epsilon^{n-2}}, \dots, \frac{1}{\epsilon}]^T$ and $W_2 = [\frac{1}{\epsilon^{n-2}}, \frac{1}{\epsilon^{n-2}}, \dots, \frac{1}{\epsilon}, \frac{1}{\epsilon^{n-1}}]^T$.  
Then 
$$
W_i^TAx = -\left(\left(\frac{1}{\epsilon}\right)^{n-1} - 1\right) x_i, \ \ 
W_i^Tb = \sum_{j=0}^{n-2}\frac{1}{\epsilon^j} = \frac{\left(\frac{1}{\epsilon}\right)^{n-1} -1 }{\frac{1}{\epsilon} -1}.
$$ 

We will show that $x_i > -1$ for all $n \geq 3$ and $i =1, \dots, n-1$.

Since $W_i \geq 0$, for $i=1, \dots, n-1$, the inequality $W_i^T A x \leq W_i^T b$ is valid for $\Delta_n$.  This directly implies that $x_i \geq -\frac{1}{\frac{1}{\epsilon}-1} > -2 \epsilon > -1$ for all $i=1, \dots, n-1$.
%
%
%
%
%
%
%

Next, notice that $[1, 1] Cx \leq [1, 1]d$ is $\frac{-1}{2\epsilon} x_n \leq 2$ which implies that $x_n \geq -4 \epsilon  > -1$ because $\epsilon < \frac{1}{4}$.

%
%
%

This combined with the inequality $C_2 x \leq d_2$ proves that $\Delta_n \cap \Z^n$ is bounded.  
Furthermore,  observe that $\frac{\epsilon}{2n}(1,1, \dots, 1) \in \intr(\Delta_n)$, which shows that $\Delta_n$ is a full dimensional simplex in $\R^n$.  

Consider now any $x \in \Delta_n \cap \Z^n \subseteq \R^n_+$.  The inequality $C_2 x \leq d_2$ implies that $x$ has support at most 2 and if $x_i > 0$ for any $i=3, \dots, n$, then $x$ has support exactly 1.  We claim that $x_i \leq 1$ for all $i=1, \dots, n$.  Clearly this is true for $i=3, \dots, n$ by the inequality $C_2 x \leq d_2$.  If $x_2 > 0$, either $x_1 = 0$, in which case then first inequality from $Ax \leq b$ implies that $x_2 \leq 1$, or $x_1 = 1$ and $x_2 = 1$, by the inequality $C_2  x \leq d_2$.
Similarly, if $x_1 > 0$, then $x_1 \leq 1$.  
This finishes the claim.

The arguments above prove that the only possible feasible integer points are $\{0, e^1, e^2, \dots, e^{n-1}, e^1 + e^2\}$.    
We see that all of the feasible integer points lie in the subspace $x_n = 0$, and thus, lie in a $(n-1)$-dimensional subspace of $\R^n$.

Finally, observe that each integer points $x \in \Delta_n \cap \Z^n$ lies the relative interior of a unique facet of $\bar \Delta_n$.  Furthermore, since $\bar \epsilon>0$ is arbitrarily small and $\Delta_n$ is a polytope, it follows, e.g.,  from~\cite[Lemma 3.6]{bhk2} that $\Delta_n \cap \Z^n = \bar \Delta_n \cap \Z^n$.   This proves that $\bar \Delta_n$ is maximial lattice free and all its integer points lie on a $(n-1)$-dimensional subspace of $\R^n$.  \mqed
\end{proof}

We next show that, in every dimension $n\geq 3$, there exist minimal inequalities that are not arbitrarily close to extreme inequalities. 

\begin{theorem}\label{thm:negative}
	Let $n \geq 3$ and let $S = b+\Z^n$ for some $b \in \R^n\setminus\Z^n$. Let $K \subseteq \R^n$ be a maximal $S$-free simplex with $0 \in \intr(K)$ and with exactly one point from $S$ on each facet.  Suppose that  the affine hull $K\cap S$ is a strict subset of $\R^n$. Then there exists an $\epsilon > 0$ such that for all minimal CCGFs $\gamma'\neq \gamma_K$ such $\|\gamma_K - \gamma'\| < \epsilon$, $\gamma'$ is not extreme.
\end{theorem}
\begin{proof}

Let $M = \sup\{\|x\|_1: x \in K\}$. Let $\epsilon > 0$ such that $K_\epsilon \cap S = K \cap S$ where $K_\epsilon := \{ x \in \R^n : \gamma_K(x) \leq 1+\epsilon M\}$. Note that $\gamma_{K_{\epsilon}} = \frac{1}{1+\epsilon M}\gamma_K$ and therefore $K_\epsilon = (1+\epsilon M)K$. So, $K_\epsilon$ is a simplex.

Let $\gamma'$ be a minimal CCGF such that $\|\gamma_K - \gamma'\| < \epsilon$. Let $K' = \{x \in \R^n : \gamma'(x) \leq 1\}$. The inequality $\|\gamma_K - \gamma'\| < \epsilon$ implies that for any $x \in K'$, $|\gamma_K(x) - \gamma'(x)| \leq \epsilon\|x\|_1$. Therefore,
$$
\gamma_K(x) \leq \gamma'(x) + \epsilon\|x \|_1 \leq 1 + \epsilon M.
$$
Thus, $K' \subseteq K_\epsilon$ and therefore, $K'$ is bounded.  Since $\gamma'$ is a minimal CCGF, it can be shown that $K'$ is a maximal $S$-free convex set. By Theorem~\ref{thm:geoCGF} part 2., $K'$ is a polytope with one point from $S$ in the relative interior of each facet.  Since $K' \cap S \subseteq K_\epsilon \cap S = K \cap S$ and $|K \cap S| = n+1$, this implies that $K'$ has $n+1$ facets and is therefore, a full-dimensional simplex. Moreover, the affine hull of $K' \cap S = K \cap S$ is a strict affine subspace of $\R^n$. By Theorem~\ref{thm:extreme-simplices}, $\gamma'$ is not extreme. \mqed
\end{proof}

\begin{proof}[Proof of Theorem~\ref{thm:contNegative}] Using simplices like the one constructed in Lemma~\ref{ex:simplex} that are shifted by $b$ to be maximal $S$-free convex sets, and Theorem~\ref{thm:negative}, one can find examples of minimal CCGFs and an $\epsilon >0$ such that all extreme CCGFs are at least $\epsilon$ distance away in the norm defined by~\eqref{eq:norm-sublinear}.  \mqed
\end{proof}

	


\clearpage
\bibliographystyle{spmpsci}
\bibliography{full-bib}


\ifipco
\else

	\newpage
	\appendix
	\noindent {\bf \LARGE Appendix}

	\section{Proof of Lemma \ref{lemma:minPiAB}} \label{app:minPiAB}
	
				By definition, $\pi_\delta(0) = 0$ and $\pi_\delta$ is periodic modulo $\Z$.  By case analysis, it can be verified that $\pi_\delta$ also satisfies the symmetry condition.
		Therefore, by Theorem~\ref{thm:minimal}, it suffices to prove that $\pi_\delta$ is subadditive in order to verify minimality.

		Let $\B = \{0, \delta, b - \delta, b, b + \delta, 1- \delta\} + \Z$ be the infinite set of breakpoints of the periodic piecewise linear function $\pi_\delta$.  Note that, by assumption that $\delta \in (0, \min\{ \frac{b}{2}, \frac{1-b}{2}\})$ and $b \in (0,1)$, we have that $0 < \delta < b-\delta < b < b + \delta < 1 - \delta < 1$.  By Lemma~\ref{lem:verts}, it suffices to check that $\Delta \pi_\delta \geq 0$ on $\verts(\Delta \P_\B)$.  
		  Every vertex of this complex satisfies either $x,y \in \B$ or $x, x+y \in \B$ or $y, x+y \in \B$.  By symmetry, we will ignore the third case.   Note that for any $x \in \B$, we have $\pi_\delta(x) \in \{0, \tfrac{1}{2}, 1\}$.  By periodicity modulo $\Z$, we can assume that $x,y \in [0,1]$ and hence $x+y \in [0,2]$.

\begin{enumerate}[{Case} 1. ]
	\item Suppose $x,y\in \B$.
	
	\begin{enumerate}[{Case 1}a. ]
		\item Suppose $\pi_\delta(x) + \pi_\delta(y) \geq 1$.\\
		 Since $\pi_\delta \leq 1$, $\Delta \pi_\delta(x,y) \geq 0$.  Furthermore, equality only holds if $\pi_\delta(x+y) = 1$, which is only attained at $x+y = b$ or $x+y = 1+b$.  Hence, in this case, equality is only attained in the case of the symmetry condition.
	
		\item Suppose $\pi_\delta(x) + \pi_\delta(y) = \tfrac{1}{2}$.\\
		Since $x,y\in \B$, this implies that either $\pi_\delta(x) = 0$ or $\pi_\delta(y) = 0$.  Without loss of generality, suppose $\pi_\delta(x) = 0$.  Then $x \in \Z$ and $x+y \equiv y \pmod 1$.   Hence, $\Delta\pi_\delta(x,y) = 0$ in this case.

		\item Suppose $\pi_\delta(x) + \pi_\delta(y) = 0$.\\
		  Since $\pi_\delta \geq 0$, it follows that $x,y \in \Z$.  Hence, $x + y \in \Z$ and $\Delta\pi_\delta(x,y) = 0$ in this case.
	
	\end{enumerate}
	
	\item Suppose $x,x+y \in \B$, $y \notin \B$.
	
	\begin{enumerate}[{Case 2}a. ]
	
		\item Suppose $\pi_\delta(x) \geq \pi_\delta(x+y)$.\\
		Since $\pi_\delta \geq 0$ and $y \notin \B$, this implies that $\Delta \pi_\delta(x,y) > 0$.

		\item Suppose $\pi_\delta(x) < \pi_\delta(x+y)$.\\
	       Since $y \notin \B$, $\pi_\delta(x) \neq 0$.  Therefore $\pi_\delta(x) = \tfrac{1}{2}$ and $\pi_\delta(x+y) = 1$.  But then $x+y = b$ or $x+y=1+b$.  Since $b - x \in \B$ and $1+b-x \in U$ for all $x \in \B$, it follows also that $y \in \B$. This is a contradiction.

	\end{enumerate}
	
\end{enumerate}

	Hence, we have established that for $(x,y) \in \verts(\Delta \P_{\B})$, $\Delta \pi_\delta \geq 0$ and thus, by Lemma~\ref{lem:verts}, $\pi_\delta$ is subadditive. Moreover, the arguments above show that $\Delta \pi_\delta(x,y) = 0$ if and only if either $x$ or $y$ in $\Z$, or $x+y \in b+\Z$. Consider the region $[\delta, 1-\delta]^2$.  Since $x = \delta, y = \delta, x  = 1-\delta$, and $y = 1-\delta$ are four lines of the polyhedral complex $\Delta \P_{\B}$, the set $[\delta, 1-\delta]^2$ is a union of faces of $\Delta \P_{\B}$.   The only vertices of $\Delta \P_{\B}$ in $[\delta, 1-\delta]^2 \cap E(\pi_\delta)$ are on the lines $x+y = b$ and $x+y = 1 +b$.  Since these lines are separated in the polyhedral complex by the line $x+y = 1$, it follows that $[\delta, 1-\delta]^2 \cap E(\pi_\delta)  = \{(x,y) : x+y = b \text{ or } x+y = 1+b\} \cap [\delta, 1-\delta]^2$.  This proves that
	$E(\pi_\delta) \subseteq E_\delta \cup E_{b} \cup E_{1+b}$.
	
	This also shows that for all other points $(x,y) \in [0,1]^2 \setminus E(\pi_\delta) \subseteq (E_\delta \cup E_{b} \cup E_{1+b})$, $\Delta \pi_\delta (x,y)$ must take a positive value.  In particular, all points in $\verts(\Delta \P_\B) \cap \left( [0,1]^2 \setminus (E_\delta \cup E_{b} \cup E_{1+b}) \right)$ have a strictly positive value for $\Delta\pi_\delta$. Since there are only finitely many such vertices, there exists $\gamma > 0$ such that $\Delta\pi_\delta(x,y) > \gamma$ for all $(x,y) \in \verts(\Delta \P_\B) \cap \left( [0,1]^2 \setminus (E_\delta \cup E_{b} \cup E_{1+b}) \right)$. By Lemma~\ref{lem:verts}, $\Delta\pi_\delta(x,y) > \gamma$ for all $(x,y) \in [0,1]^2 \setminus (E_\delta \cup E_{b} \cup E_{1+b})$. This concludes the proof.

\fi

\end{document}